\documentclass[12pt]{article}

\usepackage[lmargin = 2cm, rmargin = 2cm, bmargin = 2cm,tmargin=2cm]{geometry}

\usepackage{titlesec}

\usepackage{microtype}
  
\usepackage{enumitem}

\usepackage{xfrac}

\usepackage{subfig}

\usepackage{titlesec}

\usepackage{graphicx}

\usepackage{amssymb}

\usepackage{amsmath}

\usepackage{cancel}

\usepackage{amsthm}

\usepackage{accents}

\usepackage{wrapfig}

\usepackage{mathrsfs}

\usepackage{eufrak}

\usepackage{tikz-cd}

\usepackage{xcolor}

\usepackage{hyphenat}

\usepackage{fancyhdr}

\usepackage{url}

\usepackage{mathdots}

\usepackage{etoolbox}

\usepackage{tgtermes}

\definecolor{ggreen}{rgb}{0,0.75,0.08}

\theoremstyle{plain}

\theoremstyle{definition}

\newtheorem{theorem}{Theorem}[section]

\newtheorem{lemma}[theorem]{Lemma}

\newtheorem{corollary}[theorem]{Corollary}

\newtheorem{definition}[theorem]{Definition}

\newtheorem{question}[theorem]{Question}

\newtheorem{question*}{Question}

\newcommand{\K}{\ensuremath{\kappa}}
\newcommand{\B}{\ensuremath{\beta}}
\newcommand{\si}{\ensuremath{\sigma}}
\newcommand{\E}{\ensuremath{\varepsilon}}
\newcommand{\W}{\ensuremath{\omega}}

\newcommand{\RR}{\ensuremath{\mathbb R}}

\newcommand{\II}{\ensuremath{\mathbb I}}

\newcommand{\JJ}{\ensuremath{\mathbb J}}

\newcommand{\QQ}{\ensuremath{\mathbb Q}}

\newcommand{\NN}{\ensuremath{\mathbb N}}
\newcommand{\HH}{\ensuremath{\mathbb H}}
\newcommand{\0}{\ensuremath{\varnothing}}

\newcommand{\cD}{\ensuremath{\mathcal D}}

\newcommand{\cA}{\ensuremath{\mathcal A}}

\newcommand{\cF}{\ensuremath{\mathcal F}}
\newcommand{\cP}{\ensuremath{\mathcal P}}

\newcommand{\sE}{\ensuremath{\mathscr E}}  
\newcommand{\cE}{\ensuremath{\mathcal E}}

 \emergencystretch=\maxdimen
\begin{document}

\openup 0.6em

\fontsize{13}{5}
\selectfont

	\begin{center}\LARGE Continuum Without Non-Block Points
	\end{center}
	
	\begin{align*}
	\text{\Large Daron Anderson }  \qquad \text{\Large Trinity College Dublin. Ireland }  
	\end{align*} 
	\begin{align*} \text{\Large andersd3@tcd.ie} \qquad \text{\Large Preprint February 2016}  
	\end{align*}$ $\\

	\begin{center}
		\textbf{ \large Abstract}
	\end{center}
	
	\noindent
	For any composant $E \subset \HH^*$ and corresponding near-coherence class $\sE \subset \W^*$ we prove the following are equivalent : 
	(1) $E$ properly contains a dense semicontinuum. 
	(2) Each countable subset of $E$ is contained in a dense proper semicontinuum of $E$. 
	(3) Each countable subset of $E$ is disjoint from some dense proper semicontinuum of $E$. 
	(4) $\sE$ has a minimal element in the finite-to-one monotone order of ultrafilters. 
	(5) $\sE$ has a $Q$-point. 
	A consequence is that NCF is equivalent to $\HH^*$ containing no proper dense semicontinuum and no non-block points. 
	This gives an axiom-contingent answer to a question of the author. 
	Thus every known continuum has either a proper dense semicontinuum at every point or at no points. 
	We examine the structure of indecomposable continua for which this fails, and deduce they contain a maximum semicontinuum with dense interior.

	\section{Introduction} \noindent
	Non-block points are known to always exist in metric continua \cite{B,Leonel01}. 
	Moreover it follows from Theorem 5 of \cite{Bing01} that every point of a metric continuum is included in a dense proper semicontinuum. 
	We call a point with this property a coastal point. A coastal continuum is one whose every point is coastal.
	
	The author's investigation of whether non-metric continua are coastal began in \cite{me1}. 
	The problem was reduced to looking at indecomposable continua. 
	Specifically it was shown that every non-coastal continuum $X$ admits a proper subcontinuum $K$ such 
	that the quotient space $X/K$ obtained by treating $K$ as a single point is indecomposable and fails to be coastal 
	(as a corollary this proves separable continua are coastal).
	
	Since every indecomposable continuum with more than one composant is automatically coastal, 
	the heart of the problem rests in those indecomposable (necessarily non-metric) continua with exactly one composant. 
	We henceforth call these \textit{Bellamy continua}, after David Bellamy who constructed the first example in ZFC~\cite{one}. 
	There are very few examples known. The best-studied candidate is the Stone-\v Cech remainder $\HH^*$ of the half-line. 
	The composant number of $\HH^*$ is axiom sensitive, but under the axiom Near Coherence of Filters (NCF) the composant number is exactly one~\cite{NCF2}. 
	In the first section of this paper, we show under NCF that $\HH^*$ has neither coastal nor non-block points. 
	Thus there consistently exists a non-coastal continuum.
	
	It remains unresolved whether such a continuum can be exhibited without auxiliary axioms. 
	The only other Bellamy continua of which the author is aware arise from an inverse limit process~\cite{one,smith2,smith1}. 
	The process in fact yields a continuum with exactly two composants $-$ which are then combined by identifying a point of each. 
	The nature of this construction ensures that what used to be a composant is still a dense proper semicontinuum, 
	and so these examples are easily shown to be coastal. 
	
	Thus every known Bellamy continuum is either coastal at every point or at none. 
	One might wonder whether these are the only options. 
	This question is addressed in the paper's final section, where we show what pathology a partially coastal Bellamy continuum must display.

	\section{Notation and Terminology}
	
	\noindent 
	By a \textit{continuum} we mean a compact connected Hausdorff space. 
	We do not presume metrisability. 
	The interior and closure of a subspace $B$ are denoted $B^\circ$ and $\overline {B}$ respectively. 
	The continuum $X$ is said to be \textit{irreducible} between two points $a,b \in X$ if no proper subcontinuum of $X$ contains the subset $\{a,b\}$. 
	
	The topological space $T$ is called \textit{continuumwise connected} if for every two points $a,b \in T$ there exists a continuum $K \subset T$ with $\{a,b\} \subset K$. We also call a continuumwise connected space a \textit{semicontinuum}. Every Hausdorff space is partitioned into maximal continuumwise connected subspaces. These are called the \textit{continuum components}. When $X$ is a continuum and $S \subset X$ a subset, we call $S$ \textit{thick} to mean it is proper and has nonvoid interior. 
	The point \mbox{$p \in X$} of a continuum is called a \textit{weak cut point} to mean the subspace $X-p$ is not continuumwise connected. 
	If $a,b \in X$ are in different continuum components of $X-p$ we say that $p$ is \textit{between} $a$ and $b$ and write $[a,p,b]$.

	When $X$ is a continuum the \textit{composant} $\K(p)$ of the point $p \in X$ is the union of all proper subcontinua that include $p$. 
	Another formulation is that $\K(p)$ is the set of points $q \in X$ for which $X$ is not irreducible between $p$ and $q$. 
	For any points  $x,p \in X$ we write $\K(x;p)$ for the continuum component of $x$ in $X-p$. 
	The point $x \in X$ is called \textit{coastal} to mean that $\K(x;p)$ is dense for some $p \in X$. 
	We call $p \in X$ a \textit{non-block point} if $\K(x;p)$ is dense for some $x \in X$. 
	From the definition, a continuum has a coastal point if and only if it has a non-block point, 
	f and only if it contains a dense proper semicontinuum.

	Throughout $\W^*$ is the space of nonprincipal ultrafilters on the set $\W = \{0,1,2, \ldots\}$ 
	with topology generated by the sets $\widetilde{D} = \{\cD \in \W^* \colon D \in \cD\}$ for all subsets $D\subset \W$. 
	Likewise $\HH^*$ is the space of nonprincipal closed ultrafilters on $\HH = \{x \in \RR \colon x \ge 0\}$ 
	with topology generated by the sets $\widetilde{U} = \{\cA \in \HH^* \colon A \subset U$ for some $A \in \cA\} $ 
	for all open subsets $U \subset \HH$. 
	For background on such spaces the reader is directed to \cite{CS1} and \cite{CSbook}.

	$\HH^*$ is known to be an \textit{hereditarily unicoherent} continuum. 
	That is to say any pair of its subcontinua have  connected intersection. 
	Moreover $\HH^*$ is \textit{indecomposable}, meaning we cannot write it as the union of two proper subcontinua. 
	This is equivalent to every proper subcontinuum having void interior. 
	The composants of an indecomposable continuum are pairwise disjoint. 
	
	For any two subsets $A,B \subset \HH$ we write $A < B$ to mean $a<b$ for each $a \in A$ and $b \in B$. 
	By a \textit{simple sequence} we mean a sequence $I_n = [a_n,b_n]$ of closed intervals  of $\HH$ 
	such that $I_1 < I_2 < I_3 < \ldots$ and the sequence $a_n$ tends to infinity. 
	Suppose $\II = \{I_1,I_2, \ldots \}$ is a simple sequence. 
	For each subset $N \subset \W$ define $I_N = \bigcup \{I_n \colon n \in N\}$. 
	For each $\cD \in \W^*$ the set $\II_\cD = \bigcap \big \{ \overline {I_D} \colon D \in \cD \big \}$ is a subcontinuum of $\HH^*$. 
	These are called \textit{standard subcontinua}.  
	In case each sequence element is the singleton $\{a_n\}$ the corresponding standard subcontinuum is also a singleton, 
	called a \textit{standard point}, and we denote it by $a_\cD$.

	Throughout $\II = \{I_1,I_2, \ldots \}$ and $\JJ = \{J_1,J_2, \ldots \}$ are simple sequences. Each $I_n = [a_n,b_n]$ and $J_n = [c_n,d_n]$. For any choice of $x_n \in I_n$ the point $x_\cD$ is called a regular point of $\II_\cD$. Observe that while every regular point is standard, being regular is a relative notion. It makes no sense to say `$x$ is a regular point', only  `$x$ is a regular point of $\II_\cD$'.
	
	Standard subcontinua have been studied under the guise of ultracoproducts of intervals~\cite{PaulSSC}. This perspective makes certain properties more transparent. For example every standard-subcontinuum $\II_\cD$ is uniquely irreducible between the regular points $a_\cD$ and $b_\cD$. We call these the \textit{end points} of $\II_\cD$ and denote them by $a$ and $b$ when there is no confusion. The set $\II_\cD - \{a,b\}$ is called the interior of $\II_\cD$.
	
	There exists a natural preorder on $\II_\cD$, where  $x \sqsubseteq y$ means $y$ is between $x$ and $b$, or that every subcontinuum of $\II_\cD$ that 
	includes $b$ and $x$ must also include $y$. As per convention we write $x \sqsubset y$ to mean $x \sqsubseteq y$ but $x \ne y$. The equivalence classes of this preorder are linearly ordered and called the \textit{layers} of $\II_\cD$. 
	Layers are indecomposable subcontinua. The layer of each regular point of $\II_\cD$ is a singleton, and the set of these singletons is dense in $\II_\cD$ 
	in both the topological and order theoretic sense. For points $x,y \in \II_\cD$ we write $L^x$ and $L^y$ for their layers, and write such things as $[x,y)$ 
	to mean $ \{z \in \II_\cD \colon L^x \sqsubseteq L^z \sqsubset L^y\}$. 
	
	We define each $[x,z]$ to be intersection of all subcontinua that include the points $x,z \in \II_\cD$. By hereditary unicoherence each $[x,z]$ is a subcontinuum, called a section of $\II_\cD$. In case $x=x_\cD$ and $y=y_\cD$ are regular points then $[x,y]$ is just the standard subcontinuum $\JJ_\cD$ where each $J_n= [x_n,y_n]$. By writing $[x,y)$ as the union of all segments $[x,z]$ for $x \sqsubseteq z \sqsubset y$ we see that $[x,y)$ is a semicontinuum.

	For any function $g \colon \W \to \W$ and ultrafilter $\cD$ on $\W$ define the image 
	$g(\cD) = \{E \subset \W \colon g^{-1}(E) \in \cD\}$. 
	It can be shown $g(\cD)$ is the ultrafilter generated by $\{g(D) \colon D \in \cD\}$.  
	Suppose $\cD$ and $\cE$ are ultrafilters and $f \colon \W \to \W$ is a finite-to-one function such that 
	$f(\cD) = \cE$. Then we write $\cE \lesssim \cD$. 
	If in addition we can choose $f$ to be monotone
	we write $\cE \le \cD$. The equivalence classes of $\le$ are called \textit{shapes} of ultrafilters.
	Lemma \ref{example} is the referee's and illustrates how the partition into shapes is strictly finer than the partition
	into types. 
	
	Two free ultrafilters $\cD$ and $\cE$ are said to \textit{nearly cohere} if they have a common lower bound relative to $\le$. The principle Near Coherence of Filters (NCF) states that every two free ultrafilters nearly cohere. 
	Blass and Shelah showed this assertion is consistent relative to ZFC~\cite{NCF3} 
	and Mioduszewski showed that NCF is equivalent to $\HH^*$ being a Bellamy continuum~\cite{MiodComposants}. 
	Indeed it follows from Section 4 of Blass' ~\cite{NCF2} that the following correspondence is a bijecton  
	between the composants of $\HH^*$ and the near-coherence classes of $\W^*$: 
	Given a composant $E \subset \HH^*$ we can define the subset $\sE = \{\cD \in \W^* \colon$ some $\II_\cD$ 
	is contained in $E\}$ of $\W^*$. Likewise for each near-coherence class $\sE\subset \W^*$ 
	we can define the subset $E = \bigcup \{\II_\cD  \colon \II$ is a simple sequence and $\cD \in \sE\}$ of $\HH^*$. 
	
	

	

	\section{The Betweenness Structure of $\HH^*$}
	\noindent 
	This section establishes some tools concerning the subcontinua of $\HH^*$ for later use. Our first concerns the representation of standard subcontinua. We would like to define the shape of $\II_\cD$ to be the shape of $\cD$. To prove that this makes sense we need the following result that follows from \cite{CS1} Theorem 2.11 and the proof given for Theorem 5.3.
	
	\begin{theorem} \label{hart} Suppose that $\JJ_\cE \subset \II_\cD$. 
		Then $\cD \le \cE$. Moreover if $\cD < \cE$ as well then $\JJ_\cE$ is contained in a layer of $\II_\cD$.
	\end{theorem}
	
	\begin{lemma}
		Each standard subcontinuum has a well-defined shape.
	\end{lemma}
	
	\begin{proof}
		Suppose $\II_\cD$ and $\JJ_\cE$ are two representations of the same standard subcontinuum. That is to say $\II_\cD=\JJ_\cE$. Then we have $\JJ_\cE \subset \II_\cD$ and Theorem \ref{hart} says $\cD \le \cE$. Applying the same theorem to how $\II_\cD \subset \JJ_\cE$  shows that $\cE \le \cD$ as well. This is the definition of $\cD$ and $\cE$ having the same shape.
	\end{proof}
	
	Theorem \ref{hart} relates the $\le$ ordering to the interplay between different standard subcontinua.
	It will be helpful to know something about $\le$-minimal elements and hence about shapes of standard subcontinua that are maximal with respect to inclusion. Here the direction of the ordering is unfortunate. We mean the standard subcontinua $\II_\cD$ with $\II_\cD \subset \II_\cE$ only for $\cD$ and $\cE$ with the same shape. 
	It turns out the $\le$-minimal elements are already well-studied. 
	These ultrafilters are called $Q$-points and are usually defined as minimal elements of the $\lesssim$ ordering.
	Theorem 9.2 (b) of \cite{uff} can be used to prove the following two characterisations are equivalent.

	\begin{definition} We call $\cD \in \W^*$ a $Q$-point to mean it satisfies either (and therefore both) of the properties below.
		\begin{enumerate}
			\item Every finite-to-one function $f \colon \W \to \W$ is constant or bijective when restricted to some element of $\cD$.
			\item $\cD$ is $\lesssim$-minimal. That means $(\cE \lesssim \cD \iff \cD \lesssim \cE)$ for all $\cE \in \W^*$.
		\end{enumerate}
	\end{definition}
	
	Condition (1) shows that when $\cD$ is a $Q$-point and $f \colon \W \to \W$ finite-to-one 
	then $f(\cD)$ is either principal or is a permutation of $\cD$. 
	The next lemma proves our assertion that $Q$-points are the $\le$-minimal ultrafilters.
	
	\begin{lemma} \label{leminimal}
		The $Q$-points are precisely the $\le$-minimal elements.
	\end{lemma}
	
	\begin{proof}
		First suppose $\cD$ is a $Q$-point and that $\cE \le \cD$ for some $\cE \in \W^*$. 
		That means $\cE = f(\cD)$ for some $f \colon \W \to \W$ monotone finite-to-one. 
		There exists an element $D \in \cD$ over which $f$ is bijective. 
		The inverse $f^{-1} \colon f(D) \to \W$ is bijective monotone 
		and can be extended to a finite-to-one function on $\W$ that maps $\cE$ to $\cD$. 
		Therefore $\cD \le \cE$ as required.
		
		Now let $\cD$ be $\le$-minimal. 
		We will show it is $\lesssim$-minimal as well. 
		Suppose $\cE \lesssim \cD$ meaning $\cE = f(\cD)$ where $f \colon \W \to \W$ is finite-to-one.  
		Lemma 2.3 (2) of \cite{RBOrder} shows how to construct finite-to-one monotone functions $g$ and $h$ with $g(\cD)=h(\cE)$.
		By definition $\cE \ge h(\cE)$ and $g(\cD) \le \cD$. 
		By $\le$-minimality the second inequality implies $g(\cD) \ge \cD$. 
		Then we have $\cE \ge h(\cE) = g(\cD) \ge \cD$ and therefore $\cE \ge \cD$ which implies $\cE \gtrsim \cD$ as required.
	\end{proof}
	
	We will use the following result again and again to slightly expand a proper subcontinuum of $\HH^*$.

	\begin{lemma} \label{bulge}
		Each proper subcontinuum $K \subset \HH^*$ is contained in $\II_\cD - \{a,b\}$ for some standard subcontinuum $\II_\cD$. Moreover if $\cE$ is a $Q$-point in the near-coherence class corresponding to the composant containing $K$ we may assume without loss of generality that $\cD=\cE$.
	\end{lemma}
	
	\begin{proof}
		By Theorem 5.1 of \cite{CS1} we know $K$ is included in some standard subcontinuum $\II_\cD$. For any positive constants $\E_1, \E_2, \ldots$ define the slightly larger intervals $I'_1 = [a_1, b_1 + \E_0]$ and $I'_n = [a_n- \E_n, b_n + \E_n]$ for each $n>1$. The constants $\E_1, \E_2, \ldots$ may be chosen such that $\II' = \{I_0', I_1' , \ldots\}$ is still a simple sequence. Then the end points of $\II'_\cD$ are not elements of $\II_\cD$ and therefore not elements of $K$, as required.
		
		Suppose $\cE$ shares a near-coherence class with $\cD$. Then the proof of Theorem 4.1 of \cite{NCF2} shows $\II_\cD$ is contained in some standard subcontinuum $\JJ_{f(\cD)} = \JJ_{f(\cE)}$ for $f \colon \W \to \W$ monotone finite-to-one. But $\cE$ being a $Q$-point implies $f(\cE)=\cE$ and thus $\II_\cD \subset \JJ_\cE$. Then we may rename $\JJ$ to $\II$ and expand each interval slightly as before.
	\end{proof}
	
	We show how the ordering of layers of $\II_\cD$ relates to the weak cut point structure of $\HH^*$.

	\begin{lemma} \label{2}
		Suppose $p \in \II_\cD$ is not an end point. Then $[a,p)$ and $(p,b]$ are continuum components of $\II_\cD-p$.
	\end{lemma}

	\begin{proof} 
		Since $[a,p)$ and $(p,b]$ are semicontinua each is contained in a continuum component of $\II_\cD-p$.
		Moreover if $a$ and $b$ share a continuum component of $\II_\cD -p$ we would have $\{a,b\} \subset R \subset \II_\cD-p$ for some subcontinuum $R \subset \HH^*$. But this contradicts how $\II_\cD$ is irreducible between its endpoints. Finally observe that, by  how layers are defined, every subcontinuum of $\II_\cD$ that joins $a$ to an element of $L^p$ must contain $L^p$ and thus $p$. Therefore the continuum component of $a$ is contained in $\II_\cD - L^p \cup (p,b] = $ $[a,p)$ as required. 
		Likewise for $(p,b]$.
	\end{proof}
	
	Combining Lemmas \ref{bulge} and \ref{2} gives the following.
	
	\begin{lemma} \label{wcp}
		Every point of $\HH^*$ is between two points of its composant. In particular suppose $p \in \II_\cD$ is not an end point. Then $p$ is between $a$ and $b$. 
	\end{lemma}
	
	\begin{proof}
		Let $p \in \HH^*$ be arbitrary. By Lemma \ref{bulge} we have $p \in \II_\cD - \{a,b\}$ for some standard subcontinuum $\II_\cD$. By Lemma \ref{2} we know $p$ is between $a$ and $b$ in $\II_\cD$. But since $\HH^*$ is hereditarily unicoherent this implies that $p$ is between $a$ and $b$ in $\HH^*$ as well. Finally observe that since $\II_\cD  \subset \HH^*$ is a proper subcontinuum the points $a$, $b$ and $p$ share a composant.
	\end{proof}

	Of course if $\HH^*$ has more than one composant, every point is readily seen to be a weak cut point. But even then, it is not obvious that every point is between two points in its composant. Nor should we expect this to be true for all indecomposable continua. The result fails for the Knaster Buckethandle when removing its single end point $-$ what is left of its composant even remains arcwise connected.\\

	\begin{figure}[!h]
		\centering
		\begin{tikzpicture}[scale = 0.9]
		

		\draw (9,5) arc (0:180:9/2);
		
		\draw (6,5) arc (-180:0:9/6);
		\draw (6+1/9,5) arc (-180:0:9/6 -1/9);
		\draw (6+2/9,5) arc (-180:0:9/6 -2/9);
		\draw (6 +1/3,5) arc (-180:0:9/18 + 2/3);
		\draw (6 +2/3,5) arc (-180:0:9/18 + 1/3);
		
		\draw (7,5) arc (-180:0:9/18);
		\draw (7 -1/9,5) arc (-180:0:9/18 + 1/9 );
		\draw (7 -2/9,5) arc (-180:0:9/18 + 2/9 );
		
		\draw (8 + 2/3 + 2/9,5) arc (0:180:7/2 + 4/6 + 4/18);
		\draw (8 + 2/3 + 1/9,5) arc (0:180:7/2 + 4/6 + 2/18);
		\draw (8 + 2/3,5) arc (0:180:7/2 + 4/6 );
		
		\draw (8  + 2/9,5) arc (0:180:7/2 +4/18);
		\draw (8  + 1/9,5) arc (0:180:7/2 +2/18);
		\draw (8 + 1/3,5) arc (0:180:7/2 +2/6);
		\draw (8,5) arc (0:180:7/2);
		
		\draw (7,5) arc (0:180:5/2);

		\draw (6 + 2/3+1/9,5) arc (0:180:3/2 + 4/6 +2/18);
		\draw (6 + 2/3+2/9,5) arc (0:180:3/2 + 4/6 +4/18);
		
		\draw (6 + 2/3,5) arc (0:180:3/2 + 4/6 );
		\draw (6 + 1/3,5) arc (0:180:3/2 +2/6);
		
		\draw (6 +1/9,5) arc (0:180:3/2  +2/18);
		\draw (6 +2/9,5) arc (0:180:3/2  +4/18);
		\draw (6,5) arc (0:180:3/2);
		
		\draw(2,5) arc (-180:0:1/2);
		\draw(2 + 1/3,5) arc (-180:0:1/2 -1/3);
		\draw(2 + 1/9,5) arc (-180:0:1/2 -1/9);
		\draw(2 + 2/9,5) arc (-180:0:1/2-2/9);
		
		\draw(2/3,5 ) arc (-180:0:3/18);
		\draw(2/3 + 1/9,5 ) arc (-180:0:1/18);
		
		\draw(2/9 ,5 ) arc (-180:0:1/18);
		
		\draw(1/18 ,5 ) arc (-180:0:1/36);
		
		\draw[red, fill=red] (0,5) circle [radius=0.05];
		\node [below left] at (0,5) {\color{red} $p$};

		\end{tikzpicture}
		\caption{The end point $p$ does not cut its composant.}
	\end{figure}
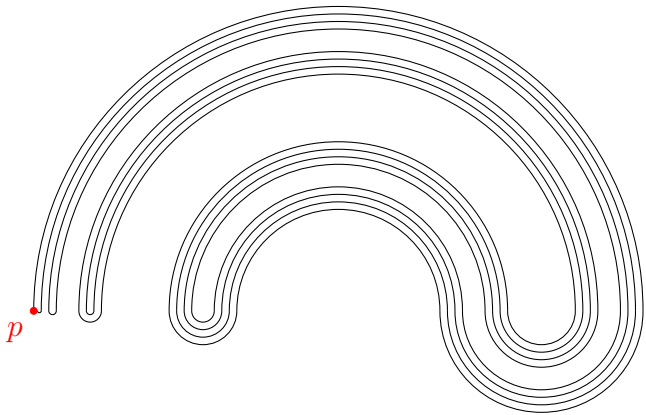
	
	We finish by giving the referee's example of two ultrafilters $\cD$ and $\cE$ such that $\cE \lesssim \cD$
	but not $\cE \le \cD$. This demonstrates how the partition of $\W^*$ into shapes is strictly finer than the partition
	into types.
	
	\begin{lemma}\label{example}
		The partition of $\W^*$ into shapes is strictly finer than the partition into types.
	\end{lemma}
	
	\begin{proof}
		
		Partition $\W$ into intervals $I_0 = \{0\}$ and $I_n = [2^{n-1}, 2^{n})$ for all $n > 0$. Define the filter $\cF$ by
		letting $F \in \cF$ exactly if $\big \{|I_n-F| \colon n \in \W \big \}$ is bounded. Choose $\cD$ as any ultrafilter
		extending $\cF$. Observe each $D \in \cD$ contains more than one element of some $I_k$; 
		otherwise $F = D^c$ is an element of $\cF$ because each $|I_n-F|$ is bounded above by $2$, 
		and this contradicts how $D$ is an ultrafilter.

		Let the permutation $\sigma$ reverse the order of elements in each $I_n$ and define 
		$\cE = \si(\cD)$. Theorem 9.2 (a) of \cite{uff} says $\cD \ne \cE$ and Theorem 9.3 says $\cD$ and $\cE$ 
		have the same type. It remains to show they have different shapes.
		
		By definition $f(\cD) = \cE$ implies $(\si^{-1}\circ f)(\cD) = \cD$.
		Then Theorem 9.2 (b) of \cite{uff} says $\si^{-1}\circ f$ is the identity, and hence $f=\si$,
		over some set $D \in \cD$. Now let $a,b \in D \cap I_k$ be distinct for some $k \in \W$. 
		It follows that $f$ reverses the order of $a$ and $b$. Therefore $f$ is not monotone
		and we cannot have $\cE \le \cD$. Therefore the shapes are different.
	\end{proof}

	\section{The Presence of $Q$-Points}

	\noindent The number of near-coherence classes of $\W^*$ (and hence composants of $\HH^*$) is axiom-sensitive. Likewise for the distribution of $Q$-points in $\W^*$. It is true in ZFC that there always exists a class without $Q$-points \cite{RBOrder}. At the same time it follows from Theorem 9.23 of \cite{CSbook} that under CH there are also $2^{\frak c}$ classes with $Q$-points. Conversely NCF implies there exists a single class and no $Q$-points \cite{NCF1}. 
	Moreover it was recently shown \cite{FmNCC} that for each $n \in \NN$ there may exist exactly $n$ classes with $Q$-points and one class without.
	
	Under the assumption that $\HH^*$ has more than one composant each point is non-block and coastal for trivial reasons: Every composant $E \subset \HH^*$ is a proper dense semicontinuum that witness how each $x \in E$ is coastal and how each $x \notin E$ is non-block. We are interested in whether this is the only reason a point can be non-block or coastal. This leads to the following definition.
	
	\begin{definition}
		The subset $P \subset X$ is called a \textit{proper non-block set} to mean that $P$ is contained in some composant $E$ of $X$, and that some continuum component of $E-P$ is dense in $X$. The subset $P \subset X$ is called a \textit{proper coastal set} to mean that $P$ is contained in a dense semicontinuum that is not a composant of $X$. Supposing the singleton $\{p\}$ is a proper non-block (coastal) set we call $p$ a proper non-block (coastal) point.
	\end{definition}
	
	It turns out the existence of proper coastal and non-block sets in a composant depends on whether the corresponding near coherence class has a $Q$-point or not. We will examine the two possibilities separately. Henceforth $\cD$ is assumed to be a $Q$-point whose near coherence class corresponds to the composant $A \subset \HH^*$. We are grateful to the referee for correcting our earlier misconception about this case, and for providing the proof of the following lemma.
	
	\begin{lemma} \label{1}For any $p \in \II_\cD  - \{a,b\}$ the semicontinuua $\K(b;p)$ and $\K(a;p)$ are dense.\end{lemma}
	
	\begin{proof}
		We only consider $\K(b;p)$ because the other case is similar. There is a regular point $q$ of $\II_\cD$ such that $p\sqsubset q \sqsubset b$. We showed in Lemma \ref{wcp} that $[p,q,b]$ which implies $\K(b;q) \subset \K(b;p)$. Therefore it suffices to show $\K(b;q)$ is dense.
		
		Assuming $\K(b;q)$ is not dense it must be nowhere dense since $\HH^*$ is indecomposable. Then $\overline{\K(b;q)}$ is a proper subcontinuum and thus by Lemma \ref{bulge} is contained in the interior of some standard subcontinuum $\JJ_\cD$ where each $J_n = [c_n,d_n]$. 
		
		It follows that $q$ and $b$ are regular points of $\JJ_\cD$ and moreover $q \sqsubset b \sqsubset d$ in $\JJ_\cD$. But then the interval $[b,d]$ of $\JJ_\cD$ witnesses how $d \in \K(b;q)$. But by construction $d \notin \K(b;q)$, a contradiction. 
	\end{proof}
	
	We can use the semicontinua constructed in Lemma \ref{1} to show any countable subset of $A$ is both proper coastal and proper non-block.
	
	\begin{theorem} \label{Q}
		Every countable $P \subset A$ is a proper non-block set and a proper coastal set.
	\end{theorem}
	
	\begin{proof}
		Let $P = \{p_1,p_2, \ldots\}$. Since all $p_i$ share a composant we can use Lemma \ref{bulge} 
		to form an increasing chain $K_1 \subset K_2 \subset \ldots$ of standard subcontinua of shape $\cD$ 
		such that each $\{p_1,p_2, \ldots, p_n\} \subset K_n$. 
		By the Baire Category theorem the union $\bigcup K_n$ is proper. 
		The complement of $\bigcup K_n$ is a nonempty $G_\delta$ set. 
		Section 1 of \cite{almostPLevy} proves the complement has nonvoid interior. 
		This implies $\bigcup K_n$ cannot be dense, 
		and so $\overline {\big ( \bigcup K_n \big )}$ is contained in the interior of some further standard subcontinuum $\II_\cD$. 
		Observe that $\II_\cD \subset A$. 
		
		There exists a regular point $q$ of $\II_\cD$ such that $a \sqsubset x \sqsubset q \sqsubset b$ for each $x \in \overline {\big ( \bigcup K_n \big )}$. 
		Moreover Lemma \ref{2} implies that $x \notin \K(b;q)$ and $b \notin \K(x;q)$. Lemma \ref{1} says the semicontinuua $\K(b;q)$ and $\K(x;q)=\K(a;q)$ are both dense. 
		Moreover the subcontinuum $\overline {\big ( \bigcup K_n \big )}$ witnesses how all $\K(p_i;q)$ coincide with each other and with $\K(x;q)$. Therefore $\K(b;q)$ witnesses how $P$ is a proper non-block set and $\K(p_1;q)$ witnesses how $P$ is a proper coastal set.  
	\end{proof}
	
	One can ask whether Theorem \ref{Q} can be strengthened by allowing the set $P$ to have some larger cardinality. In particular 
	we might look for the least cardinal $\eta(A)$ such that every element of $\{P \subset A \colon |P|<\eta(A)\}$ 
	is a proper coastal set and a proper non-block set.

	To see the number $\eta(A)$ is at most $2^{\aleph_0}$  consider the family $P$ of all standard points $a_\cD$ for each $a_n$ rational. 
	The family has cardinality $|\QQ^ \NN| = {\aleph_0}^{\aleph_0} = 2^{\aleph_0}$ and is easily seen to be dense in $A$. 
	Therefore $P$ cannot be a proper coastal set. Moreover under the Continuum Hypothesis $\aleph_0$ has successor 
	$2^{\aleph_0}$ and so it is consistent that $\eta(A) = 2^{\aleph_0}$. 
	\begin{question}
		Is it consistent that there exists a composant $A \subset \HH^*$ where the corresponding near coherence class has $Q$-points and $\eta(A) < 2^{\aleph_0}$?
	\end{question}
	
	\begin{question}
		Is it consistent that there exist composants $A,A' \subset \HH^*$ where the corresponding near coherence classes have $Q$-points and $\eta(A) \ne \eta(A')$?
	\end{question}

	Next we will treat the case when the composant $B \subset \HH^*$ corresponds to a near coherence class without a $Q$-point. 
	We remark it is consistent for the two composants $A$ and $B$ to exist simultaneously. For example in the model presented in \cite{FmNCC}
	or indeed any model where the Continuum Hypothesis holds.

	The outcome for $B$ is the complete opposite to that for $A$ $-$ the composant $B$ has neither proper coastal points nor proper non-block points. Our main tool to prove this is the following lemma, which is alluded to in the literature $-$ for example in \cite{RBOrder} $-$ but for which we have been unable to find a complete proof.
	
	\begin{lemma} \label{chain}
		$B$ is the union of an increasing chain of proper indecomposable subcontinua.
	\end{lemma}
	
	\begin{proof}
		By Zorn's lemma there exists a maximal increasing chain $\cP$ of proper indecomposable subcontinua in $B$. 
		We claim the union $\bigcup \cP$ is dense. For otherwise by Lemma \ref{bulge} we have 
		$\bigcup \cP \subset \II_\cD \subset B$ for some standard subcontinuum $ \II_\cD$. 
		Then since $\cD$ is not a $Q$-point it is not $\le$-minimal by Lemma \ref{leminimal}. 
		Therefore we have $\cE < \cD$ for some $\cE \in \W^*$. 
		It follows from Theorem \ref{hart} that $\II_\cD$ is contained in a layer of $\II_\cE$. 
		But that layer is a proper indecomposable subcontinuum and so can be added as the new top element of 
		$\bigcup \cP$, contradicting how the chain is maximal. We conclude $\bigcup \cP$ is dense.
		
		Compactness implies there is some $x \in \bigcap \cP$. To prove $\bigcup \cP = B$ we take \mbox{$b \in B$} to be arbitrary.
		There exists a standard subcontinuum $L$ with $\{x,b\} \subset L$. We have already shown $\bigcup \cP$ is dense.
		That means some $P \in \cP$ is not contained in $L$. By Theorem 5.9 of \cite{CS1} we have $L \subset P$.
		This implies $b \in P \subset \bigcup \cP$ as required.\end{proof}
	
	\begin{theorem} \label{noQ}
		$B$ has no proper coastal points and no proper non-block points.
	\end{theorem}
	
	\begin{proof}
		Let $\cP$ be an increasing chain of proper indecomposable subcontinua with union $B$.
		Recall each $P \in \cP$ is nowhere-dense. Now let $S \subset B$ be an arbitrary proper semicontinuum. 
		That means we can fix a point $y \in S$ and write $S = \bigcup_{x \in S} S(x)$ 
		where each $S(x)$ is a subcontinuum containing $\{x,y\}$.

		Choose any point $b \in (B-S)$. There exists $P$ such that $\{b,y\} \in P \subset \cP$. 
		The point $b$ witnesses how $P \not \subset S(x)$ and so Theorem 5.9 of \cite{CS1} 
		implies that $S(x) \subset P$. But since $x \in S$ is arbitrary this implies $S \subset P$. 
		Therefore $S$ is nowhere-dense.
		
		We conclude that $B$ contains no proper dense semicontinuum. 
		It follows $B$ has no proper coastal points and therefore no proper non-block points.
	\end{proof}

	Under NCF there are no $Q$-points. In this case Theorem \ref{noQ} tells us $\HH^*$ has no proper non-block points. But NCF is also equivalent to $\HH^*$ having exactly one composant, and that implies any non-block points that exist must be proper. So we have the stronger result.
	
	\begin{theorem} \label{big}
		(NCF) $\HH^*$ lacks coastal points and non-block points.
	\end{theorem}

	Every separable continuum and {\it a fortiori} every metric continuum has two or more non-block points. The author has asked whether the separability assumption can be dropped. Theorem \ref{big} gives an axiom-contingent answer. 
	
	\begin{corollary} \label{con}There consistently exists a continuum without non-block points.\end{corollary}
	
	Whether Corollary \ref{con} can be proved in ZFC alone is currently unresolved. 
	One possible line-of-attack to the problem is as follows: 
	Observe that every composant of an hereditarily indecomposable continuum is the union of the same sort of chain as described in Lemma \ref{chain}. 
	From here a similar proof to Theorem \ref{noQ} shows hereditarily indecomposable continua lack proper non-block points. 
	Thus an hereditarily indecomposable Bellamy continuum would be an example of a continuum without non-block points. 
	Smith has found some obstacles to constructing such a beast \cite{SmithHIRemainderProd,SmithHILexProd,SmithHISouslinArcsIL,SmithHISouslinProd}. 
	But we have so far no reason to believe one exists.
	
	We can combine the main results of this section into two sets of equivalences. The first set looks at each composant separately.
	
	\begin{corollary}
		The following are equivalent for any composant $E \subset \HH^*$ and corresponding near coherence class $\sE \subset \W^*$.
		\begin{enumerate}
			
			\item Some point of $E$ is proper non-block (coastal).
			
			\item Every point of $E$ is proper non-block (coastal).
			
			\item Some countable subset of $E$ is proper non-block (coastal).
			
			\item Every countable subset of $E$ is proper non-block (coastal).
			
			\item $\cE$ has a $Q$-point.
			
			\item $\cE$ has a $\le$-minimal element.
			
		\end{enumerate}
	\end{corollary}
	
	The second set looks at $\HH^*$ as a whole.
	
	\begin{corollary}
		The following are equivalent.
		\begin{enumerate}
			\item NCF
			\item $\HH^*$ has exactly one composant
			\item $\HH^*$ lacks coastal points and non-block points.
		\end{enumerate}
	\end{corollary}
	
	In particular we have that $-$ regardless of the model $-$ it can only be the case that either every point of $\HH^*$ is non-block or none are. This observation motivates the next and final section.

	\section{Partially Coastal Bellamy Continua}
	\noindent
	Thus far every Bellamy continuum has proved to be either coastal at every point or at no points. This section examines the remaining case. 
	Henceforth $H$ is some fixed Bellamy continuum. We will make frequent use of the fact that every semicontinuum $S \subset H$ is either dense or nowhere-dense. 
	Under the assumption that there is a coastal point $x \in H$  and a non-coastal point $y \in H$, 
	this section investigates how badly behaved $H$ must be. 
	
	Our description is in terms of thick semicontinua. Recall the semi- continuum $S \subset H$ is called thick to mean it is proper and has nonvoid interior. 
	Every indecomposable metric continuum has more than one composant and so cannot contain a thick semicontinuum. 
	It is unknown whether the result generalises $-$ no known Bellamy continuum contains a thick semicontinuum. 
	Thus the following lemma explains our failure to provide a concrete example for the continuum $H$.
	
	\begin{lemma} \label{yes} $H$ contains a thick semicontinuum. Moreover every dense proper semicontinuum in $H$ is thick.\end{lemma}
	
	\begin{proof} Let $S$ be an arbitrary dense proper semicontinuum. At least one exists to witness how the point $x \in H$ is coastal.
		Since $H$ has one composant there is a proper subcontinuum $K \subset H$ with $\{x,y\} \subset K$. Then $K \cup S$ is also a dense semicontinuum.
		But since $y$ is non-coastal we must have $K \cup S = H$. Therefore $S$ contains the open set $H-K$ and hence has nonvoid interior.
	\end{proof}
	
	\begin{lemma} $H$ has a thick semicontinuum that contains every other thick semicontinuum.\end{lemma}
	
	\begin{proof} Let $S$ be the thick semicontinuum found in Lemma \ref{yes}. 
		Every other thick semicontinuum $M \subset H$ is dense, and since $S$ has interior this implies $S \cup M$ is a semicontinuum. 
		Moreover $y \notin M$ since $y$ is non-coastal and hence $S \cup M$ is proper. It follows the union of $S$ with all possible choices for $M$
		is the maximum among thick semicontinua of $H$.
	\end{proof}
	
	Henceforth we will fix $S \subset H$ to be the maximum thick semicontinuum.
	
	\begin{lemma} \label{cc}$S$ is a continuum component of $H-p$ for each $p \in H - S$.\end{lemma}
	
	\begin{proof} We know $S$ is contained in some continuum component $C$ of $H - p$. 
		But since a continuum component is a semicontinuum this implies $C \subset S $ by maximality of $S$ and 
		therefore $C = S$. \end{proof}

	\begin{lemma} \label{compl} $H - S^\circ$ is a subcontinuum and one of two things happens. 
		\begin{enumerate} 
			\item The thick semicontinuum $S$ is open
			\item $H-S^\circ$ is indecomposable with more than one composant
		\end{enumerate}
	\end{lemma}
	\begin{proof}
		Observe that by boundary-bumping the arbitrary point $p$ is in the closure of every continuum component of $H-p$. 
		Therefore any union of continuum components of $H-p$ has connected closure. 
		Lemma \ref{cc} says $S$ is a continuum component of $H-p$. 
		Therefore $\overline {(H-p-S)} = H - S^\circ$ is connected and hence a continuum. Call this continuum $B$.
		
		There is a partition $B = A \cup C$ where $A = H-S $ is the complement of $S$ and $C = B \cap S$ consists of the points of $S$ outside its interior. 
		Since $S$ is proper $A$ is nonempty. If we assume $S$ is not open then $C$ is nonempty as well.
		
		We will demonstrate that $B$ is irreducible between each $a \in A$ and each $c \in C$. 
		Since $A$ and $C$ form a partition this will imply they are both unions of composants. 
		In particular $B$ will have two disjoint composants making it indecomposable.
		
		Now suppose $E \subset B$ is a subcontinuum that meets each of $A$ and $C$. 
		Since $E$ meets $C = B \cap S$ we know that $S \cup E$ is continuumwise connected. 
		But since $E$ meets $A = H-S$ we know $S \cup E$ is strictly larger than $S$. 
		By assumption $S$ is a maximal thick semicontinuum. So the only option is that $S \cup E = H$. 
		
		In particular $S \cup E$ contains $A = H-S$. 
		This implies $A \subset E$ and since $E$ is closed $\overline {A} \subset E$. But by definition $\overline {A} = \overline {(H-S)} = H-S^\circ = B$. 
		We conclude $E = B$ as required.\end{proof}
	
	The previous lemma showed that $H-S^\circ$ is a subcontinuum. 
	Recall that, since $H$ is indecomposable, its every subcontinuum has void interior. This gives us the corollary.
	
	\begin{corollary} \label{di} $S$ has dense interior.\end{corollary} 
	
	Now we are ready to identify the coastal points of $H$.
	
	\begin{lemma} $H$ is coastal exactly at the points of $S$. \end{lemma}
	
	\begin{proof} 
		
		Since $S$ is the maximum dense proper semicontinuum, each of its points is coastal. 
		Now let $x \in H$ be an arbitrary coastal point. That means $x$ is an element of 
		some proper dense semicontinuum which is, by definition, contained in $S$. This implies $x \in S$ as required. \end{proof}
	
	We can summarise the progress made in this section in the theorem.
	
	\begin{theorem} \label{summary}Suppose $S$ is the set of coastal points of the Bellamy continuum $H$ and $\0 \ne S \ne H$. Then $S$ is a semicontinuum with dense interior and contains every semicontinuum of $H$ with nonempty interior. Moreover $H-S^\circ$ is a subcontinuum and one of the below holds.
		\begin{enumerate} 
			\item The semicontinuum $S$ is open
			\item $H-S^\circ$ is an indecomposable continuum with more than one composant
	\end{enumerate}\end{theorem}
	
	We conclude this section with a remark on the three classes of continua. Continua of the first class are coastal at every point $-$ for example all metric or separable continua \cite{me1}. Continua of the second class have no coastal points $-$ for example $\HH^*$ under NCF.
	
	Continua of the third (and possibly empty) class are coastal only at the points of some proper subset. However, as Theorem \ref{summary} tells, these continua have the extra property of being \textit{simultaneously coastal}. That means the set of coastal points knits together into a dense proper semicontinuum that simultaneously witnesses the coastal property for each of its points.
	
	\section*{Acknowledgements}
	This research was supported by the Irish Research Council Postgraduate Scholarship Scheme grant number GOIPG/2015/2744. The author would like to thank Professor Paul Bankston and Doctor Aisling McCluskey for their help in preparing the manuscript. We are grateful to the referee for their time and attention, and for correcting a major oversight in the paper's first iteration.

\end{document}